\documentclass[12pt,reqno,oneside]{amsart}
\usepackage{amsmath,amssymb,amsfonts,amsthm,graphicx,cite,epsfig,makecell,array,mathrsfs}
\usepackage[varg]{pxfonts}
\usepackage{wrapfig}
\usepackage{enumitem}
\usepackage{graphicx}
\usepackage{subcaption}
\usepackage{epstopdf,footnote,lipsum}
\usepackage{eucal}

\usepackage[colorlinks=true,linkcolor=blue,citecolor=red]{hyperref}
\numberwithin{equation}{section}

\usepackage{lineno}

\newtheorem{theorem}{Theorem}[section]

\newtheorem{lemma}[theorem]{Lemma}

\theoremstyle{definition}
\newtheorem{definition}[theorem]{Definition}

\theoremstyle{remark}
\newtheorem{remark}[theorem]{Remark}
\numberwithin{equation}{section}

\begin{document}


\title[ Fekete-Szeg\"{o} inequalities for Gamma-Bazilevic functions ... ]
{  Gamma-Bazilevic functions related with generalized telephone numbers }

\author{Gangadharan Murugusundaramoorthy$^1$ ,kaliappan Vijaya$^1$, Hijaz Ahmad$^2,*$}
\address[G.Murugusundaramoorthy, K.vijaya]{$^1$Department of Mathematics, School of Advanced Science\\ \,\,Vellore Institute of Technology\\ Vellore - 632014,TN., India. \\
}
\email{gmsmoorthy@yahoo.com; kvijaya@vit.ac.in}
\address[Hijaz Ahmad ]{$^2,*$ Corresponding author ,Section of Mathematics, International Telematic University Uninettuno, Corso Vittorio Emanuele II, 39,00186 Roma, Italy\\
}
\email{hijaz555@gmail.com}

 \maketitle

\begin{abstract}
The purpose of this paper is to consider coefficient estimates in a
class of functions $\mathfrak{G}_{\vartheta}^{\kappa}(\mathcal{X},\varkappa)$ consisting of analytic
functions $f$ normalized by $f(0)=f'(0)-1=0$\ in the open unit disk  $\Delta=\{ z:z\in \mathbb{C}\quad \text{and}\quad \left\vert
z\right\vert <1\}$ subordinating generalized telephone numbers,
 to derive certain  coefficient estimates $a_2,a_3$ and Fekete-Szeg\"{o} inequality for $f\in\mathfrak{G}_{\vartheta}^{\kappa}(\mathcal{X},\varkappa)$. A similar results have been  done for the function $ f^{-1} $ and $\log\dfrac{f(z)}{z}.$Similarly application of our results to certain functions defined by using convolution products with a normalized
 analytic function is given, and in particular we state
Fekete-Szeg"{o} inequalities for 
subclasses  described through Poisson Borel and Pascal distribution series.
\end{abstract}
\section{Introduction, Definitions and Preliminaries}
\subsection{Generalized telephone numbers (GTN)- $\mathfrak{T}_\varkappa(n)$:} The classical telephone numbers (TN), distinguished as involution numbers , are exact by the recurrence relation $$\mathfrak{T} (n) = \mathfrak{T} (n - 1) + (n- 1)\mathfrak{T} (n - 2) \quad for\quad n \geq 2$$ with
$$\mathfrak{T} (0) = \mathfrak{T} (1) = 1.$$ .
Associates of those numbers with symmetric groups have been perceived first time in 1800 by  Heinrich August Rothe, who piercing out that $\mathfrak{T} (n)$ is the variety of involutions (self-inverse permutations)in the symmetric groups
(see,as an instance,\cite{SC,DEK}). Due to the fact involutions resembles to standard younger tableaux,
it is vibrant that the $n^ {th}$ involution quantity is consistently the range of young tableaux on the set ${1, 2,..., n}$ (for information,see \cite{BS}). it's far really worth mentioning according to John Riordan, the above recurrence relation, in reality, yields the range of production styles in a cellphone machine
with $n$ subscribers(see \cite{JR}). In 2017, Wlochand Wolowiec-Musial \cite{AWM} familiarized
generalized telephone numbers (GTN) via the succeeding recursion
 $$ \mathfrak{T} (\varkappa, n) = \varkappa \mathfrak{T}(\varkappa, n - 1) + (n - 1)\mathfrak{T} (\varkappa, n- 2)\quad n \geq 0\quad{\text and}\quad\varkappa \geq 1$$ with  $$\mathfrak{T} (\varkappa, 0) = 1, \mathfrak{T} (\varkappa, 1) = \varkappa,$$ and deliberated some properties. In 2019, Bednarz
and Wolowiec-Musial\cite{UB} presented a new generalization of TN by
$$\mathfrak{T}_\varkappa(n) = \mathfrak{T}_\varkappa(n - 1) +\varkappa(n- 1)\mathfrak{T}_\varkappa(n- 2),\quad n \geq 2\quad{\text and}\quad\varkappa \geq 1$$ with  $$\mathfrak{T}_\varkappa(0) = \mathfrak{T}_\varkappa(1) = 1.$$
They provided the generating function, straight formula
and matrix generators for these numbers. Moreover, they acquired clarifications and
proved  few properties of these numbers related with congruence’s.
These days, they resulting the exponential generating function and the summation method $\mathfrak{T}_\varkappa(n)$ as follows:
\[ e^{x+\varkappa \frac{x^2}{2}}=\sum_{n=0}^{\infty}\mathfrak{T}_\varkappa(n) \frac{x^n}{n!}\quad (\varkappa\geq 1)
\]
As we are able to observe,if $\varkappa = 1,$ then we achieve classical telephone numbers $\mathfrak{T} (n).$
Clearly,  $\mathfrak{T}_\varkappa(n)$ is for a few values of $n$ as

\begin{center}
\begin{enumerate}
\item $\mathfrak{T}_{\varkappa}(0) =\mathfrak{T}_{\varkappa} = 1,$\\
\item $\mathfrak{T}_{\varkappa}(2) = 1 + \varkappa,$\\
\item $\mathfrak{T}_{\varkappa}(3) = 1 + 3\varkappa$ \\
\item $\mathfrak{T}_{\varkappa}(4) = 1 + 6\varkappa + 3\varkappa ^2$\\
\item $\mathfrak{T}_{\varkappa}(5) = 1 + 10\varkappa + 15\varkappa^2$\\
\item  $\mathfrak{T}_{\varkappa}(6) =
1 + 15\varkappa + 45\varkappa^2 + 15\varkappa^3.$
\end{enumerate}
\end{center}
In \cite{ED}, for$z\in
{\Delta}:=\left\{ z:z\in \mathbb{C}\quad \text{and}\quad \left\vert z\right\vert
<1\right\}
$ an unit disc, Deniz consider $$\mathcal{X}(z) := e^{(z+\varkappa \frac{z^2}{2})}=1+z+\frac{1+\varkappa}{2}z^{2}+\frac{1+3\varkappa}{6}z^{3}+\frac{3\varkappa^{2}+6\varkappa+1}{24}z^{4}+\frac{1 + 10\varkappa + 15\varkappa^2}{120}z^{5}+\cdots.
$$
. 
\subsection{ Subclasses of analytic functions $\mathfrak{A}$ } Denote by $\mathfrak{A}$  the class of analytic functions as given  by
\begin{equation} \label{e1.3}
  f(z)=z + \sum_{n=2}^{\infty}a_n z^n, z\in \Delta  .
\end{equation}
Also, denote $\mathfrak{S}$ be the
subclass of $\mathfrak{A}$ together with univalent functions in $\Delta$ with  $f(0)=0=f'(0)-1$ . In \cite{robert}, Robertson introduced the following classes:
 \begin{equation}\label{star}
\mathfrak{S}^*=\{f\in \mathfrak{S}: \Re\Big(\frac{zf'(z)}{f(z)}\Big)>0, \;  \quad (z\in \Delta)\}
  \end{equation}
   and
    \begin{equation}\label{convex}
 \mathfrak{C}=\{ f\in \mathfrak{S}:  \Re\Big( \frac{(zf'(z))'}{f'(z)}\Big)>0, \quad (z\in \Delta)\}.
     \end{equation}

      The class of functions fulfilling the analytic standards given  by  \eqref{star} and \eqref{convex} are referred to as as \textit{starlike} and \textit{convex} functions in $\Delta$ respectively.

Assuming $f_1, f_2 \in  \mathfrak{S} $ then we are saying that the function$f_1$
is subordinate to $f_2$ if  there exists a Schwarz charcterestic $\varpi(z)$, analytic in
$\Delta$ with
$ \varpi(0) = 0 \quad \text{and} \quad  \left| \varpi(z)\right|<1 \quad
(z \in \Delta), $ such that$ f_1(z) = f_1( \varpi (z)) \quad (z \in \Delta). $ 
 We
denote this subordination by means of
\[ f_1 \prec f_2 \quad \text{or} \quad f_1(z) \prec f_2(z) \quad (z \in
\Delta). \] Particularly, if  $f_2$ is univalent in $\Delta,$ the above
subordination is equal to$$ f_1(0) = f_2(0)\quad and \quad f_1(\Delta) \subset f_2(\Delta).$$
Let
\begin{equation}\label{3Int}
\mathfrak {S}^*(\psi)=\{ f  \in  \mathfrak{S}:  \frac{zf'(z)}{f(z)}\prec \psi(z)\}
\end{equation}
where $\psi(z)=1+m_1 z +m_2 z^2 +m_3 z^3 +\cdots.,  m_1>0.$
By varying the function $\psi$, several familiar classes can be obtained as illustrated below:
\begin{enumerate}
\item For $\psi =\frac{1+Az}{1+Bz}~(-1\leq B< A\leq 1)$, we get the class $\mathfrak {S}^*%
\left( A,B\right)$, see \cite{jan}.
Also by fixing $A=1-2\alpha$ and $B=-1$, we have $\mathfrak {S}^*
\left( \alpha \right) =\mathfrak{S}^{\ast }\left( 1-2\alpha ,-1\right) $ \cite{robert}.

\item In \cite{ronning} ,  by taking  $\psi =1+\frac{2}{\pi ^{2}}\left( \log \frac{1+\sqrt{z}}{1-\sqrt{z}%
}\right) ^{2}$, a new class  was defined and studied.
\item  Assuming  $\psi(z)
=:z+\sqrt{1+z^2},\ \ z\in {\Delta},$  Raina and Sokol\cite{RKSJ}and Sokol and Thomas \cite{sokol1} extensively discussed the geometric properties for $f\in{\mathfrak S}_S^*(\psi)$ \\
\item In \cite{8.2a},the class  ${\mathfrak S}_L^*(\psi)=\{ f  \in  \mathfrak{S}:  \frac{zf'(z)}{f(z)}\prec \sqrt{1+z}\},$ was studied and  further studied in \cite{mohsin}.
\item The class  ${\mathfrak S}_C=\{ f  \in  \mathfrak{S}:  \frac{zf'(z)}{f(z)} \prec  1+\frac{4}{3}z+\frac{2}{3}z^2 \}$  was introduced in and investigated in \cite{car ,arifc}.
\item In \cite{7a,19}  the authors defined and
discussed the class  ${\mathfrak S}_e^*(\psi)=\{ f  \in  \mathfrak{S}:  \frac{zf'(z)}{f(z)}\prec e^{z}\}.$
\item For $\psi =1+\sin \left( z\right)$, the class is denoted by $\mathfrak {S}^*_{\sin }$, see \cite{CHO} .
\item For $\psi =\cosh \left( z\right)$, the class is denoted by $\mathfrak {S}^*
_{\cosh },$ see \cite{arif}.
\end{enumerate}
\par Lately, for  $\vartheta \geqq 0,\,\, \kappa \geqq 0$  and $ f \in
\mathfrak{A}$ Fitri
 and  Thomas
 \cite{Guo} studied the a new class $ {G}(\vartheta,\kappa)$ which holds the following:
\small{
\begin{eqnarray}
\Re \left \{ \left[ \frac{zf'(z)}{(f(z)^{1-\kappa}z^\kappa} +\frac{zf''(z)} {f'(z)} +(\kappa-1) \left( \frac{zf'(z)}
{f(z)}-1 \right) \right]^\vartheta \left[\frac{zf'(z)}{(f(z)^{1-\kappa}z^\kappa}\right]^{1-\vartheta} \right\}> 0
\end{eqnarray}
}
and discussed its characterization results.
 Inspired fundamentally by the aforesaid works (see
\cite{SME,murugu,RKSJ,RBS,GMDT}), and recent work of Murugusundaramoorthy and Vijaya \cite{GMKV}, in this paper  first time we describe a new class $\mathfrak{G}_{\vartheta}^{\kappa}(\mathcal{X},\varkappa)$  as  given in Definition \ref{defn} which coalesces the many new  subclasses of $\mathfrak S^* $ and
$\mathfrak C $ in association with GTN.  First,we shall find estimations of $a_2$ and $a_3$ for
 $f\in\mathfrak{G}_{\vartheta}^{\kappa}(\mathcal{X},\varkappa)$ of the form \eqref{e1.3}  also for
$ f^{-1} \in \mathfrak{G}_{\vartheta}^{\kappa}(\mathcal{X},\varkappa)$ and $log\frac{f(z)}{z}.$
Further we prove the Fekete-Szeg\"{o} inequality  for a general class . Additionally we confer certain applications of our consequences  by way of convolution  to certain classes defined through Poisson, Borel and Pascal distributions .

Now, we define the following class $\mathfrak{G}_{\vartheta}^{\kappa}(\mathcal{X},\varkappa)$ :

\begin{definition}\label{defn}
  For $ \vartheta \geqq 0,\, \kappa \geqq 0$ a function
$f\in\mathfrak{A}$ is in the class $\mathfrak{G}_{\vartheta}^{\kappa}(\mathcal{X},\varkappa)$ if
\begin{eqnarray}
 &&\left[ \frac{zf'(z)}{(f(z))^{1-\kappa}z^\kappa} +\frac{zf''(z)} {f'(z)} +(\kappa-1) \left( \frac{zf'(z)}
{f(z)}-1 \right) \right]^\vartheta \left[\frac{zf'(z)}{(f(z))^{1-\kappa}z^\kappa}\right]^{1-\vartheta}\nonumber \\ \displaystyle &&\prec e^{(z+\varkappa \frac{z^2}{2}
)}=:\mathcal{X}(z); \;z=re^{i\theta}\in \Delta.
\end{eqnarray}
\end{definition}
By specializing the parameters $\vartheta $ and $\kappa$ we state the subsequent new subclasses of $\mathfrak{S}$ as illustrated under which are not discussed sofar for for functions associated with GTN:
\begin{remark}\label{GME-REM}
\
\\
\begin{enumerate}
  \item $\displaystyle \mathfrak{G}_{0}^{0}(\mathcal{X},\varkappa)\equiv \mathfrak{S}^{\ast}(\mathcal{X},\varkappa)=\left \{f\in\mathfrak{A}: \frac{zf'(z)} {f(z)}\prec e^{(z+\varkappa \frac{z^2}{2})} ;z=re^{i\theta}\in \Delta.\right\}$
  \item \begin{eqnarray*}\displaystyle \mathfrak{G}^{0}_{\vartheta}(\mathcal{X},\varkappa)\equiv \mathfrak{G}_\kappa(\mathcal{X},\varkappa)=\left \{ f\in\mathfrak{A}: \left[\frac{zf'(z)} {f(z)}\right]^{1-\vartheta}\left[\frac{(zf'(z))'} {f'(z)}\right]^\vartheta\prec e^{(z+\varkappa \frac{z^2}{2})} ;z=re^{i\theta}\in \Delta\right\}\nonumber\end{eqnarray*}
  \item $\mathfrak{G}_{0}^{1}(\mathcal{X},\varkappa)\equiv \mathfrak{C}(\mathcal{X},\varkappa)=\left \{ f\in\mathfrak{A}: \frac{(zf'(z))'} {f'(z)}\prec e^{(z+\varkappa \frac{z^2}{2})} ;z=re^{i\theta}\in \Delta\right\}\nonumber $\\
      \item $\mathfrak{G}^{\kappa}_{0}(\mathcal{X},\varkappa)\equiv \mathfrak{B}_\kappa(\mathcal{X},\varkappa)=\left \{ f\in\mathfrak{A}: \frac{zf'(z)} {f(z)} \left(\frac{f(z)}{z} \right)^{\kappa}\prec e^{(z+\varkappa \frac{z^2}{2})} ;z=re^{i\theta}\in \Delta\right\}\nonumber $
          \\
      \item $\mathfrak{G}^{1}_{1}(\mathcal{X},\varkappa)\equiv \mathfrak{R}(\mathcal{X},\varkappa)=\left \{ f\in\mathfrak{A}: f'(z)+\frac{zf''(z)} {f'(z)}\prec e^{(z+\varkappa \frac{z^2}{2})} ;z=re^{i\theta}\in \Delta\right\}\nonumber $
\end{enumerate}
\end{remark}
\subsection{A set of Lemmas}Recent years Fekete-Szeg\"{o} results for the
class of starlike, convex and various other subclasses of analytic functions, were studied
interested reader may refer  to\cite{GMKV,GMDT,HMSM1,HM,HMT} . In this article also we aimed to discuss the Fekete-Szeg\"{o} problem for $f\in\mathfrak{G}_{\vartheta}^{\kappa}(\mathcal{X},\varkappa)$ in association with GTN.
To prove our predominant end result, we take into account the following:
\par Let $\mathbf{P}$ denote class of functions with positive real part in $\Delta$, and be assumed as  of the form $$p(z)=1+c_1z+c_2z^2+\cdots, z\in \Delta.$$ 
\begin{lemma}\cite{mamin}\label{maminlemma}
If $p\in \mathbf{P}$,
then
$$
|c_2-v c_1^2| \leqq \left\{  \begin{array}{lll}
-4v+2, & \mbox{ if } & v\leqq 0, \\[3mm]
2,     & \mbox{ if } & 0 \leqq v\leqq 1,\\[3mm]
4v-2,  & \mbox{ if } & v\geqq 1. \end{array}\right.
$$

While $v<0$ or $v>1$, the equality holds if and only if $p_1(z)$ is $\displaystyle
\frac{1+z}{1-z}$ or one in all its rotations.  If $0<v<1$, then equality holds if and
only if $p_2(z)=\displaystyle \frac{1+z^2}{1-z^2}$or considered one of its rotations . If
$v=0$, the equality holds if and only if
$$ p_3(z) =\left( \frac{1}{2}+\frac{1}{2}\eta \right) \frac{1+z}{1-z}
+ \left( \frac{1}{2}-\frac{1}{2}\eta \right) \frac{1-z}{1+z} \quad (0\leqq \eta
\leqq 1)
$$ or one in every of its rotations. If $v=1$, the equality holds if and best if $p_1$ is the
reciprocal of one of the functions such that the equality holds when
$v=0$.
\end{lemma}
We also need the following:
\begin{lemma}\label{lem1}\cite{r1b}
If $p\in \mathbf{P}$,
then 
\begin{align}
\mid c_{n} \mid&\leq 2\;\;\forall   n\geq1 \qquad {\rm and }\qquad
| c_{2}-\frac{c_{1}^{2}}{2}|\leq 2- \frac{|c_{1}|^{2}}{2}.\nonumber
\end{align}\end{lemma} 
\begin{lemma}\cite{r1d} \label{lemma2}
If $p\in \mathbf{P}$,
 and $ v\in \mathbb{C} $ ( complex numbers),
then
$$
|c_2- v c_1^2| \leqq 2 \max(1,|2v-1|).
$$
The result is sharp for the functions
$$
 p_1(z)=\frac{1+z^2}{1-z^2}, \quad  p_2(z)=\frac{1+z}{1-z}.
$$
\end{lemma}

\begin{lemma}\label{lem2}\cite{KeoghMer69}
If $p\in \mathbf{P}$,
then $ \hbar\in\mathbb{ C} $,
\begin{align}
\Big| c_{2}-\hbar \frac{c_{1}^{2}}{2}\Big|&\leq \; \max\{2,2|\hbar-1|\}=\left\{
    \begin{array}{ll}
      2, & \hbox{$0\leq \hbar\leq 2 $;} \\
      2|\hbar-1|, & \hbox{elsewhere.}\nonumber
    \end{array}
  \right.
\end{align}
The result is sharp for the functions defined by $ p_1(z)=\frac{1+z^{2}}{1-z^{2}}$ or $ p_2(z)=\frac{1+z}{1-z}$.\\
\end{lemma}

\section{ Coefficient Estimate}
\noindent By making use of the Lemma $\ref{maminlemma}$, we prove the following:
\begin{theorem}\label{gmsth1}
 Let \,$ \vartheta \geq 0 \,\, \text{and}\,\,   \kappa \geq 0.$ If  $f\in\mathfrak{G}_{\vartheta}^{\kappa}(\mathcal{X},\varkappa)$  be as in  $(\ref{e1.3})$ , then
\begin{eqnarray*}
|a_2|  &\leqq & \frac{1}{(1+\vartheta)(1+\kappa)},\\
|a_3| & \leqq & \frac{1}{(1+2\vartheta)(1+2\kappa)}\max \{1,\big|\displaystyle \frac
{\left (M\kappa^2+S\kappa+Q \right)
}{\left((1+\vartheta)(1+\kappa)\right)^2}+\frac{1+\varkappa}{2}\big|\}\end{eqnarray*} where
\begin{equation*}\label{m}
  M =  \vartheta^2-\vartheta+1;\quad
S = 2\vartheta^2-4\vartheta+1 ; \quad
 Q = \vartheta^2-7\vartheta-2.
\end{equation*}
\end{theorem}
These  consequences are sharp.
\begin{proof}Define  $P(z)\in \mathbf{P}$ by
\begin{eqnarray}\label{p3}
 P(z):&=&\frac{1+w(z)}{1-w(z)}=1+c_1z+c_2z^2+\cdots.\notag\\ \text{it is easy to see that}\notag\\
w(z)&=&\frac{P(z)-1}{P(z)+1}\notag\\ &=&\frac{1}{2}\left[c_{1}z+\left(c_{2}-\frac{c_{1}^2}{2}\right)z^{2}
+\left(c_{3}-c_{1}c_{2}+\frac{c_{1}^{3}}{4}\right)z^{3}+\cdots\right]. \end{eqnarray}
Since $w(z)$ is a Schwarz function, we see that $\Re(P(z))>0$ and
$P(0)=1$.Thus
\begin{eqnarray}\label{p5}
\mathcal{X}(w(z))&=& e^{(\frac{P(z)-1}{P(z)+1}+\varkappa \frac{[\frac{P(z)-1}{P(z)+1}]^2}{2}
)} \notag\\
&=&1+\frac{c_{1}}{2}z
  +\Big(\frac{c_{2}}{2}+\frac{(\varkappa-1)c_{1}^{2}}{8}\Big)z^{2}+\Big(\frac{c_{3}}{2}
  +(\varkappa-1)\frac{c_{1}c_{2}}{4}+\frac{(1-3\varkappa)}{48}c_1^3\Big)z^{3}+.....\notag\\
  \end{eqnarray}
If $f\in \mathfrak{G}_{\vartheta}^{\kappa}(\mathcal{X},\varkappa)$, then there is a Schwarz function $w(z) $,
analytic in $\Delta$ with $ w(0)=0$ and $|w(z)|<1$ in $ \Delta$ such that
\begin{eqnarray} \label{p2}
~&&  \left[ \frac{zf'(z)}{(f(z))^{1-\kappa}z^\kappa} +\frac{zf''(z)} {f'(z)} +(\kappa-1) \left( \frac{zf'(z)}
{f(z)}-1 \right) \right]^\vartheta \left[\frac{zf'(z)}{(f(z))^{1-\kappa}z^\kappa}\right]^{1-\vartheta}\nonumber\\&& =\mathcal{X}(w(z))\nonumber\\&&=e^{(w(z)+\kappa \frac{[w(z)]^2}{2}
)}.
\end{eqnarray}

For given $f(z)$ of the form $(\ref{e1.3})$,  a computation indicates that
\[ \frac{zf'(z)}{f(z)}= 1+a_2z+(2a_3-a_2^2)z^2+(3a_4+a_2^3-3a_3a_2)z^3+\cdots.\]
Similarly we have
\[ 1+\frac{zf''(z)}{f'(z)}= 1+2a_2z+(6a_3-4a_2^2)z^2+\cdots.\]
Let us outline  $W(z)$ by
 \begin{eqnarray}\label{WZ}
 W(z):&=&\left[ \frac{zf'(z)}{(f(z))^{1-\kappa}z^\kappa} +\frac{zf''(z)} {f'(z)} +(\kappa-1) \left( \frac{zf'(z)}
{f(z)}-1 \right) \right]^\vartheta \left[\frac{zf'(z)}{(f(z))^{1-\kappa}z^\kappa}\right]^{1-\vartheta} \nonumber\\
\end{eqnarray}
An easy computation indicates that
 \begin{eqnarray}W(z)
  &=&1+(1+\vartheta)(1+\kappa)a_2
  z + \displaystyle (1+2\vartheta)(2+\kappa)a_3 z^{2}\notag \\
&\quad \quad+&  \left (\kappa^2(\vartheta^2-\vartheta+1)+\kappa(2\vartheta^2-4\vartheta+1)+(\vartheta^2-7\vartheta-2)  \right)a_2^{2} z^2+\cdots\notag\\
&=&1+b_1z+b_2z^2+\cdots.\label{p4}
\end{eqnarray}

Now by \eqref{p5} and \eqref{p4},
\begin{equation}\label{b1b2}
  b_1=\frac{c_1}{2} \quad \quad {\mbox{and}} \quad \quad
b_2=\frac{c_{2}}{2}+\frac{(\varkappa-1)c_{1}^{2}}{8}.
\end{equation}

In view of the equation  (\ref{p4}) and \eqref{b1b2}, we see that
\begin{eqnarray}
b_1 &=&(1+\vartheta)(1+\kappa)a_2,\label{GMB1}\\
b_2 &=&  \displaystyle (1+2\vartheta)(2+\kappa)a_3 \notag \\
&\quad \quad+&  \left (\kappa^2(\vartheta^2-\vartheta+1)+\kappa(2\vartheta^2-4\vartheta+1)+(\vartheta^2-7\vartheta-2)  \right)a_2^{2}\label{GMB2}
\end{eqnarray}
or equivalently, we have
\begin{eqnarray}
a_2  &=& \frac{c_1}{2(1+\vartheta)(1+\kappa)},\label{gmsa2}\\
a_3  &=& \frac{1}{(1+2\vartheta)(1+2\kappa)} \left( \frac{c_2}{2} -\frac{c_1^2}{8}
\left[1-\varkappa- \displaystyle \frac
{2\left (\kappa^2(\vartheta^2-\vartheta+1)+\kappa(2\vartheta^2-4\vartheta+1)+(\vartheta^2-7\vartheta-2)  \right)
}{\left((1+\vartheta)(1+\kappa)\right)^2}\right]\right)\notag\\
&=& \frac{1}{2(1+2\vartheta)(1+2\kappa) } \left( c_2 -
\frac{c_1^2}{4}\left[1-\varkappa- \displaystyle \frac
{2\left (\kappa^2(\vartheta^2-\vartheta+1)+\kappa(2\vartheta^2-4\vartheta+1)+(\vartheta^2-7\vartheta-2)  \right)
}{\left((1+\vartheta)(1+\kappa)\right)^2}\right]\right).\notag\\\label{gmsa3}
\end{eqnarray}
For brevity we let
\begin{equation}\label{m}
  M =  \vartheta^2-\vartheta+1;\quad
S = 2\vartheta^2-4\vartheta+1 ; \quad
 Q = \vartheta^2-7\vartheta-2
\end{equation}

\begin{eqnarray}\label{a3new}
  a_3 &=& \frac{1}{2(1+2\vartheta)(1+2\kappa) } \left( c_2 -
\frac{c_1^2}{4}\left[1-\varkappa- \displaystyle \frac
{2\left (\kappa^2(\vartheta^2-\vartheta+1)+\kappa(2\vartheta^2-4\vartheta+1)+(\vartheta^2-7\vartheta-2)  \right)
}{\left((1+\vartheta)(1+\kappa)\right)^2}\right]\right).\notag\\
  &=& \frac{1}{2(1+2\vartheta)(1+2\kappa) } \left( c_2 -
\frac{c_1^2}{4}\left[1-\varkappa- \displaystyle \frac
{2\left (M\kappa^2+S\kappa+Q \right)
}{\left((1+\vartheta)(1+\kappa)\right)^2}\right]\right).
\end{eqnarray}
Now by taking absolute on\eqref{gmsa2}  and applying Lemma \ref{lem1}, we get
\[|a_2 | \leqq \frac{1}{(1+\vartheta)(1+\kappa)}
\]
and  by taking absolute on \eqref{a3new} and applying  Lemma \ref{lemma2}
 we have
\begin{eqnarray*}
   |a_3|&\leqq& \frac{1}{(1+2\vartheta)(1+2\kappa)}\\&\times&\max \{1,\big|2\times\frac{1}{4}\left[1-\varkappa- \displaystyle \frac
{2\left (M\kappa^2+S\kappa+Q \right)
}{\left((1+\vartheta)(1+\kappa)\right)^2}\right]-1\big|\}\\
   &=& \frac{1}{(1+2\vartheta)(1+2\kappa)}\max \{1,\frac{1}{2}\big|-\left(\displaystyle \frac
{2\left (M\kappa^2+S\kappa+Q \right)
}{\left((1+\vartheta)(1+\kappa)\right)^2}\right)-1-\varkappa\big|\}\\
&=& \frac{1}{(1+2\vartheta)(1+2\kappa)}\max \{1,\big|\displaystyle \frac
{\left (M\kappa^2+S\kappa+Q \right)
}{\left((1+\vartheta)(1+\kappa)\right)^2}+\frac{1+\varkappa}{2}\big|\}.
   \end{eqnarray*}
   The first two bounds  are sharp for the function $f:
\Delta\longrightarrow \mathbb{C}$ given by
\begin{align*}
f(z)&=\int_{0}^{z}\mathcal{X}(t)dt\\&=\int_{0}^{z}e^{t+\frac{\varkappa
t^{2}}{2}}dt\\&=z+\frac{z^{2}}{2}+\frac{1+\varkappa}{6}z^{3}+\frac{1+3\varkappa}{24}z^{4}+\frac{3\varkappa^{2}
+6\varkappa+1}{120}z^{5}+\cdots.
\end{align*}
Here we have $b_1=1$ and $b_2=1/2.$
By using \eqref{GMB1} and \eqref{b1b2}, we get $$|a_2|=\frac{1}{(1+\vartheta)(1+\kappa)}$$ and
again by using\eqref{b1b2}, \eqref{GMB2} we have
 $$c_{2}+\frac{(\varkappa-1)c_{1}^{2}}{4}= (1+2\vartheta)(1+2\kappa)a_3+\left (\displaystyle (\kappa^2(\vartheta^2-\vartheta+1)+\kappa(2\vartheta^2-4\vartheta+1)+(\vartheta^2-7\vartheta-2)  \right)a_2^{2}
.$$
Substituting  for $a_2=\frac{1}{(1+\vartheta)(1+\kappa)}$  simple calculation and taking absolute value gives
$$|a_3|=\frac{1}{2(\vartheta+2)(1+2\kappa)}\left|\frac
{\vartheta^{2} + \vartheta- \displaystyle 2(\vartheta+3)\kappa -2
}{(\left(1+\vartheta)(1+\kappa)\right)^2}-\varkappa-1\right|.$$
 \end{proof}
 By way of assuming $\vartheta = 0
\,\, \text{and}\,\,   \kappa \geqq 0 $ we state the following
\begin{remark}
 If  $f\in \mathfrak{G}_{\kappa}(\mathcal{X},\varkappa)$ and as in $(\ref{e1.3})$  then
\begin{eqnarray*}
|a_2|  &\leqq & \frac{1}{1+\kappa},\\
|a_3| & \leqq & \frac{1}{2(1+2\kappa)}\max \{1,\big|\frac
{\kappa^2+8\kappa +3
}{2(1+\kappa)^2}+\varkappa\big|\}=\frac{1}{2(1+2\kappa)}\left(\frac
{\kappa^2+8\kappa +3
}{2(1+2\kappa)(1+\kappa)^2}+\varkappa\right).\end{eqnarray*}
\end{remark}
By way of fixing $\vartheta = 0 =\kappa
$ we state the following
\begin{remark}
  If  $f\in\mathfrak{S}^{*}(\mathcal{X},\varkappa)$ and as assumed in $(\ref{e1.3})$  then
\begin{eqnarray*}
|a_2|  \leqq  1,\qquad {\text and~}\qquad
|a_3|  \leqq \frac{1}{2} \max \{1,\big|\frac
{3
}{2}+\varkappa\big|\}=\frac
{1
}{2}\big(\frac
{3
}{2}+\varkappa\big).\end{eqnarray*}
\end{remark}
By way of assuming $\vartheta = 0
\,\, \text{and}\,\,   \kappa=1 $ we state the following
\begin{remark}
  If  $f\in\mathfrak{C}(\mathcal{X},\varkappa)$ and as in $(\ref{e1.3}),$ then
\begin{eqnarray*}
|a_2|  \leqq  \frac{1}{2},\qquad {\text and}\qquad
|a_3|  \leqq  \frac{1}{6}\max \{1,\big|\frac
{1
}{2}+\varkappa\big|\}=\frac{1}{6}\big(\frac
{1
}{2}+\varkappa\big).\end{eqnarray*}
\end{remark}By way of letting $\kappa= 0
$ we state the following
\begin{remark}
  If  $f\in\mathfrak{G}_{\vartheta}^{0}(\mathcal{X},\varkappa)=\mathfrak{B}_{\vartheta}(\mathcal{X},\varkappa)$ and as in $(\ref{e1.3})$ , then
\begin{eqnarray*}
|a_2|  &\leqq & \frac{1}{1+\vartheta},\\
|a_3| & \leqq &\frac{1}{\vartheta+2}\max \{1,\frac{1}{2}\big|\left(\frac
{\vartheta^{2} + \vartheta-2
}{\left(1+\vartheta\right)^2}\right)-1-\varkappa\big|\}=\frac{1}{2(1+\vartheta)}\left(\frac
{\vartheta +3
}{(1+\vartheta)^2}+\varkappa\right).\end{eqnarray*}
\end{remark}By way of fixing $\vartheta = 1
\,\, \text{and}\,\,   \kappa =0 $ we state the following
\begin{remark}
If  $f\in \mathfrak{R}(\mathcal{X},\varkappa)$ given by $(\ref{e1.3})$  then
\begin{eqnarray*}
|a_2|  &\leqq & \frac{1}{2},\\
|a_3| & \leqq & \frac{1}{3}\max \{1,\frac{1}{2}\big|1+\varkappa\big|\}=\frac{1}{6}\left(1+\varkappa\right).\end{eqnarray*}
\end{remark}

\section{Fekete-Szeg\H{o} type problems}

\begin{theorem}\label{th1}
 Let \,$ 0 \leqq \mu \leqq1,\,\,   \vartheta \geqq 0
\,\, \text{and}\,\,   \kappa \geqq 0.$ If  $f\in\mathfrak{G}_{\vartheta}^{\kappa}(\mathcal{X},\varkappa)$ and assumed as in  $(\ref{e1.3})$ then
\begin{eqnarray*}
|a_3-\mu a_2^2| &\leqq & \left\{  \begin{array}{lll}\displaystyle \frac{1}{2 \mathbf{L}}
\left( 1+\varkappa+\frac{\aleph} { \mathbf{W}^{2}} \right), & \mbox{ if } &
 \mu \leqq \sigma_1, \\[5mm] \displaystyle
\frac{1}{\mathbf{L}}, & \mbox{ if } & \sigma_1\leqq \mu \leqq \sigma_2, \\ [5mm]
\displaystyle \frac{-1}{2 \mathbf{L}} \left( 1+\varkappa+\frac{\aleph} { \mathbf{W}^{2}} \right),
& \mbox{ if } &
\mu \geqq \sigma_2,\\
\end{array}\right.
\end{eqnarray*}
where, for convenience,
\begin{eqnarray}
\nonumber \sigma_1 = \frac{(\varkappa-1) \mathbf{W} ^2 +2( M\kappa^{2}+ S\kappa+Q)
}
{2 \mathbf{L} };
\sigma_2 = \frac{\varkappa \mathbf{W} ^2 +2( M\kappa^{2}+ S\kappa+Q)
}
{2 \mathbf{L} };
\end{eqnarray}
\begin{equation} \label{2.1}
  \aleph    :=
2( M\kappa^{2}+ S\kappa+Q)
 - 2 \mu \mathbf{L},
\end{equation}
\begin{equation}\label{xi}
\mathbf{L}:= (1+2\vartheta)(1+2\kappa),
\end{equation}
and
\begin{equation}\label{tau}
\mathbf{W}:= (1+\vartheta)(1+\kappa)
\end{equation} and $M,S,Q$ are assumed as in \eqref{m}.
\end{theorem}

\begin{proof}
 Now by using \eqref{gmsa2} and \eqref{gmsa3} , we get
\begin{eqnarray*}
a_3-\mu a_2^2
&=&\frac{1}{2(1+2\vartheta)( 1+ 2 \kappa) }\left(c_2-\frac{c_1^2}{4}\times \right.\\&~& \left.\left[1-\varkappa-\frac{2( M\kappa^{2}+ S\kappa+Q)
 - 2 \mu (1+2\vartheta) ( 1+ 2\kappa)}{\left( (1+\vartheta)( 1 + \kappa)\right)^2}\right]\right)\\
 &=&\frac{1}{2(\vartheta+2)( 1+ 2 \kappa) }\left(c_2-v c_1^2\right)
\end{eqnarray*}
where
\begin{eqnarray*} v:&=& \frac{1}{4}\left[1-\varkappa-\frac{2( M\kappa^{2}+ S\kappa+Q)
 - 2 \mu (1+2\vartheta) ( 1+ 2\kappa)}{\left( (1+\vartheta)( 1 + \kappa)\right)^2}\right]\\
&=& \frac{1}{4}\left[1-\varkappa-\frac{2( M\kappa^{2}+ S\kappa+Q)
 - 2 \mu \mathbf{L}}{\mathbf{W}^2}\right].
\end{eqnarray*} The proclamation of  Theorem \ref{th1} now trails by applying
Lemma~\ref{maminlemma}.

\end{proof}

Using  Lemma $\ref{lemma2}$, we directly find the following:

\begin{theorem} \label{hms} Let $   0 \leqq  \vartheta
\leqq 1,\text{and}\,\,  0 \leqq \kappa \leqq 1.$
If $ f \in \mathfrak{G}_{\vartheta}^{\kappa}(\mathcal{X},\varkappa),$ then for 
$\mu\in\mathbb{C}$, we have
\begin{eqnarray*}
|a_3-\mu a_2^2|&\leq& \frac{1}{(\vartheta+2)(1+2\kappa)}\\&\times& \max\left\{1, \frac{1}{2}\left| -1-\varkappa- \frac{2( M\kappa^{2}+ S\kappa+Q)
 - 2 \mu (1+2\vartheta) ( 1+ 2\kappa)}{\left( (1+\vartheta)( 1 + \kappa)\right)^2}\right| \right\}\\&\leq& \frac{1}{\mathbf{L}}\max\left\{1, \frac{1}{2}\left| 1+\varkappa+\frac{2( M\kappa^{2}+ S\kappa+Q)
 - 2 \mu \mathbf{L}}{\mathbf{W}^2}\right| \right\}.
\end{eqnarray*}
\end{theorem}



 \section{Coefficient inequalities for  $ f^{-1}$}
\begin{theorem}\label{thm1d}
 If $ f \in\mathfrak{G}_{\vartheta}^{\kappa}(\mathcal{X},\varkappa)$ and $ f^{-1}(w)=w+\sum \limits_{n=2}^{\infty}d_{n}w^{n}$ is the inverse function of $ f $ with $ |w|<r_{0}$ where $ r_{0} $ is greater than the radius of the Koebe domain of the class $ f \in \mathfrak{G}_{\vartheta}^{\kappa}(\mathcal{X},\varkappa)$,we have
 \begin{align*}
|d_{2}|&\leq\frac{1}{2(1+\vartheta)(1+\kappa)}\\
|d_{2}|&\leq \frac{1}{2\mathbf{L} }max\;\Big\{1,\mid\frac{-(1+\varkappa)\mathbf{W}^2-2\left (M\kappa^2+S\kappa+Q \right)+ 4\textbf{L}
}{2\mathbf{W}^2}\mid\Big\}.
\end{align*}
For  any $ \hbar\in \mathbb{C} $, we have
\begin{align}
\mid d_{3}-\hbar d_{2}^{2}\mid \leq \frac{1}{\mathbf{L}}max\;\Big\{1, \mid\frac{(1+\varkappa)\mathbf{W}^2+2\left (M\kappa^2+S\kappa+Q \right)+ 2\textbf{L}(\hbar-2)
}{2\mathbf{W}^2}\mid\Big\}\label{7.1}
\end{align} where $M,S,Q$ are assumed as in \eqref{m} and $\mathbf{L,W}$ are as in \eqref{xi} and \eqref{tau}.
\end{theorem}
\begin{proof}
As
\begin{equation}
 f^{-1}(w)=w+\sum \limits_{n=2}^{\infty} d_{n}w^{n}\label{7.2}
\end{equation}
it can be understood that
\begin{equation}
f^{-1}(f(z))=f\{f^{-1}(z)\}=z.\label{7.3}
\end{equation}
From  \eqref{e1.3} and \eqref{7.3}, we get
\begin{equation}
f^{-1}(z + \sum \limits_{n=2}^{\infty} a_{n}z^{n})=z.\label{7.4}
\end{equation}
Considering \eqref{7.3} and \eqref{7.4},  you could attain
\begin{equation}
z+(a_{2}+d_{2})z^{2}+(a_{3}+2a_{2}d_{2}+d_{3})z^{3}+.........=z.\label{7.5}
\end{equation}
By relating the coefficients of $ z $ and $ z^{2}$ from the expression \eqref{7.5},
 it can be understood that

\begin{align}
d_{2}&=-a_{2}\label{7.6}\\
d_{3}&=2a_{2}^{2}-a_{3}.\label{7.7}
\end{align}
From relations \eqref{gmsa2},\eqref{gmsa3},\eqref{7.6} and \eqref{7.7}
\small{{
\begin{align}
d_{2}&=-\frac{c_1}{2(1+\vartheta)(1+\kappa)}=-\frac{c_1}{2\mathbf{W}};\label{7.8}
\end{align}
}}
The estimate $|d_3|$ follows at once by fixing $\mu =2$ in Fekete–Szegö theorem \ref{hms}. For any $ \hbar \in \mathbb{C}$, consider
\begin{align}
d_{3}-\hbar d_{2}^{2}=-\frac{1}{2\mathbf{L}}\Big(c_{2}-\frac{(1-\varkappa) \mathbf{W}^2 -2\left (M\kappa^2+S\kappa+Q \right)+ 2\textbf{L}(2-\hbar)
}{4\mathbf{W}^2}c_1^2\Big)\label{7.10}
\end{align}
Taking absolute value of \eqref{7.10} and by using making use of Lemma \ref{lemma2} to the right hand aspect of \eqref{7.10}, you can still derive the end result as in \eqref{7.1}.
\end{proof}
\begin{remark}
For the function classes given in Remark \ref{GME-REM}, you can still easily state above result analogues to Theorem \ref{thm1d} by means of fixing the parameters suitably in Theorem \ref{thm1d} it's miles worthy to notice they are new and no longer been studied thus far in association  with telephone numbers.
\end{remark}
\section{ logarithmic coefficients of $f$}
Then, {\em the logarithmic coefficients} $\gamma_n$ of $f\in\mathfrak{S}$ are demarcated with the assistance of the resulting series expansion:
\begin{equation}\label{defnlogcoeff}
\log\dfrac{f(z)}{z}=2\sum\limits_{n=1}^{\infty}{\gamma_n(f)z^n},\;z\in\mathbb{U}.
\end{equation}
Recall that we can redraft \eqref{defnlogcoeff} in the series form as follows:
\begin{align*}
2\sum\limits_{n=1}^{\infty}\gamma_nz^n=&a_2z+a_3z^2+a_4z^{3}+\cdots-\dfrac{1}{2}[a_2z+a_3z^2+a_4z^{3}+\dots]^2\\
&+\dfrac{1}{3}[a_2z+a_3z^2+a_4z^{3}+\cdots]^3+\cdots,\;z\in\mathbb{U},
\end{align*}
and seeing the coefficients of $z^n$ for $n=1,2$, it follows that

\begin{equation}\label{formga1-3}
\left\{
\begin{array}{l}
2\gamma_1=a_2,\\
2\gamma_2=a_3-\dfrac{1}{2}a_2^2,\\
\end{array}
\right.
\end{equation}
\begin{theorem}\label{loggmsth1}
 Let \,$ \kappa \geqq 0, \,\,\,\,\vartheta \geqq 0 \text{\,\, and \,\,\,\, if }
 f\in\mathfrak{G}_{\vartheta}^{\kappa}(\mathcal{X},\varkappa)$, be as assumed in \eqref{e1.3}
  then
\begin{eqnarray*}
|\gamma_1|  &\leqq & \frac{1}{2(1+\vartheta)(1+\kappa)},\\
|\gamma_2| & \leqq & \frac{1}{\mathbf{L}}\max\left\{1, \frac{1}{2}\left| 1+\varkappa+\frac{2( M\kappa^{2}+ S\kappa+Q)
 - \mathbf{L}}{\mathbf{W}^2}\right| \right\}\end{eqnarray*} where\end{theorem}
\begin{proof}We note first that since $a_2  = \frac{c_1}{2(1+\vartheta)(1+\kappa)},$\eqref{gmsa2} and $ |c_1| \leq 2, $ the inequality $|\gamma_1| $  is insignificant. The result for $|\gamma_2|$ trails by taking $\mu = \frac{1}{2}$ in the Fekete–Szegö theorem \ref{hms}. \end{proof}
\section{Application to functions  based on convolution}
The class of univalent functions were intensively studied with the aid of numerous researchers in
exclusive prospective involving certain distributions namely  Borel , Binomial, Poisson, logarithm, Pascal,  hypergeometric.
In this section based on convolution, we defined a new generalized class and discuss the Fekete-Szeg\H{o} type problems.
In addition we discuss these results based on certain probability distribution series.
\par Let
$
\displaystyle \wp(z)=z+\sum_{n=2}^\infty \wp_nz^n,\quad (\wp_n>0)
$ and $ f\in{\mathfrak{A}} $  then
\begin{eqnarray}\label{F}
\displaystyle \mathcal{F}(z)&=&(f*\wp)(z)=z+\sum_{n=2}^\infty \wp_na_n z^n\notag\\
&=&z+\wp_2a_2z^2+\wp_3a_3z^3+\cdots
\end{eqnarray}
We define the class $ \mathfrak{G}_{\vartheta,\,\kappa}
^\wp(\mathcal{X},\varkappa)$ in the following way:
$$
\mathfrak{G}_{\vartheta,\,\kappa}
^\wp: =\{ f\in{\mathfrak{A}} \quad {\mbox{and}} \quad
\mathcal{F}(z) \in
 \mathfrak{G}_{\vartheta}^{\kappa}(\mathcal{X},\varkappa)\}
$$
 where $ \mathfrak{G}_{\vartheta}^{\kappa}(\mathcal{X},\varkappa)$ is given by Definition \ref{defn} ,
 $$\left[ \frac{z\mathcal{F}'(z)}{(\mathcal{F}(z))^{1-\kappa}z^\kappa} +\frac{z\mathcal{F}''(z)} {\mathcal{F}'(z)} +(\kappa-1) \left( \frac{z\mathcal{F}'(z)}
{\mathcal{F}(z)}-1 \right) \right]^\vartheta \left[\frac{z\mathcal{F}'(z)}{(\mathcal{F}(z))^{1-\kappa}z^\kappa}\right]^{1-\vartheta}\prec\mathcal{X}(z)$$
 \par Now, we obtain the coefficient
estimate for $f\in\mathfrak{G}_{\vartheta,\,\kappa}
^\wp(\mathcal{X},\varkappa)$, from the corresponding estimate for
$f\in\mathfrak{G}_{\vartheta}^{\kappa}(\mathcal{X},\varkappa)$. Applying Theorem~\ref{th1} for the function \eqref{F},
we get the following Theorems $\ref{gms}$ and \ref{th11}
after an obvious  change of the parameter $\mu$.Our main result is the following:
\begin{theorem} \label{gms} Let $   0 \leqq \kappa \leqq 1,\text{and}\,\, 0 \leqq  \vartheta
\leqq 1.$
If $ f \in \mathfrak{G}^\wp_{\vartheta,\kappa}(\mathcal{X},\varkappa),$ then for
$\mu\in \mathbb{C}$, we have
\begin{eqnarray*}
|a_3-\mu a_2^2|&&= \frac{2}{(1+2\vartheta)(1+2\kappa)\wp_3}\\&& max\left\{1, \frac{1}{2}\left| -1-\varkappa+\frac {2\left (M\kappa^2+S\kappa+Q \right)
}{\left((1+\vartheta)(1+\kappa)\wp_2\right)^2}+
  \frac{2 \mu (\vartheta+2)(1+2\kappa)\wp_3}
{((1+\vartheta)(1+\kappa)\wp_2)^2}\right| \right\},
\end{eqnarray*} where $M,S,Q$ are assumed as in \eqref{m} .
\end{theorem}

\begin{proof}
For $f(z)\in \mathfrak{G}^\wp_{\vartheta,\kappa}(\mathcal{X},\varkappa)$ snd $(f\ast\wp)(z)=\mathcal{F}(z)$ given by \eqref{F} we have
\begin{eqnarray} \label{pe2}
P(z):&=&\left[ \frac{z\mathcal{F}'(z)}{(\mathcal{F}(z))^{1-\kappa}z^\kappa} +\frac{z\mathcal{F}''(z)} {\mathcal{F}'(z)} +(\kappa-1) \left( \frac{z\mathcal{F}'(z)}
{\mathcal{F}(z)}-1 \right) \right]^\vartheta \left[\frac{z\mathcal{F}'(z)}{(\mathcal{F}(z))^{1-\kappa}z^\kappa}\right]^{1-\vartheta}\nonumber\\
  &=&1+b_1z+b_2z^2+\cdots.
\end{eqnarray}
Continuing as in Theorem \ref{gmsth1},  we get
\begin{eqnarray}
   \begin{array}{ll} P(z)  =1+(1+\vartheta)(1+\kappa)\wp_2 a_2
  z + \displaystyle (1+2\vartheta)(2+\kappa)\wp_3 a_3 z^{2} \\
\quad \quad+  \left (\kappa^2(\vartheta^2-\vartheta+1)
+\kappa(2\vartheta^2-4\vartheta+1)+(\vartheta^2-7\vartheta-2)  \right)\wp^2_2 a_2^{2} z^2+\cdots.\label{PMG}
\end{array}
\end{eqnarray}
From \eqref{GMB1}- \eqref{gmsa3}
and from this equation\eqref{PMG}, we obtain
\begin{eqnarray}&& a_2= \frac{c_1}{2(1+\vartheta)(1+\kappa)\wp_2}\label{PMG1}\\
&&a_3 = \frac{1}{2(1+2\vartheta)(1+2\kappa)\wp_3 } \left( c_2 -
\frac{c_1^2}{4}\left[1-\varkappa- \displaystyle \frac
{2\left (M\kappa^2+S\kappa+Q \right)
}{\left((1+\vartheta)(1+\kappa)\wp_2\right)^2}\right]\right)\label{PMG2}.\end{eqnarray}
\begin{eqnarray*}&&a_3-\mu a_2^2\\&&=
\frac{1}{2(1+2\vartheta)(1+2\kappa)\wp_3 } \left( c_2 -
\frac{c_1^2}{4}\left[1-\varkappa- \displaystyle \frac
{2\left (M\kappa^2+S\kappa+Q \right)
}{\left((1+\vartheta)(1+\kappa)\wp_2\right)^2}\right]\right)\\&&-
\mu\frac{c^2_1}{4(1+\vartheta)^2(1+\kappa)^2\wp^2_2}\\&&=
\frac{1}{2(1+2\vartheta)(1+2\kappa)\wp_3 }\left[ c_2 -\frac{c_1^2}{4}
\left(1-\varkappa- \displaystyle \frac
{2\left (M\kappa^2+S\kappa+Q \right)
}{\left((1+\vartheta)(1+\kappa)\wp_2\right)^2}\right.\right.\\&&+\left.\left.
\mu\frac{2(1+2\vartheta)(1+2\kappa)\wp_3}{(1+\vartheta)^2(1+\kappa)^2\wp^2_2}\right)\right].
\end{eqnarray*}
Consequently, by applying Lemma~\ref{lemma2} we get the desired result. The result is sharp  by assuming
\[ \frac{z\mathcal{F}'(z)} {\mathcal{F}(z)} \left(\frac{\mathcal{F}(z)}{z} \right)^{\vartheta} + \kappa
\left[ 1+ \frac{z\mathcal{F}''(z)} {\mathcal{F}'(z)} -\frac{z\mathcal{F}'(z)} {\mathcal{F}(z)}
+ \vartheta \left( \frac{z\mathcal{F}'(z)}
{\mathcal{F}(z)}-1 \right) \right]= \mathcal{X}(z)\] and

\[ \frac{z\mathcal{F}'(z)} {\mathcal{F}(z)} \left(\frac{\mathcal{F}(z)}{z} \right)^{\vartheta} + \kappa
\left[ 1+ \frac{z\mathcal{F}''(z)} {\mathcal{F}'(z)} -\frac{z\mathcal{F}'(z)} {\mathcal{F}(z)}
+ \vartheta \left( \frac{z\mathcal{F}'(z)}
{\mathcal{F}(z)}-1 \right) \right] = \mathcal{X}(z^{2})\]
\end{proof}

 \begin{theorem}\label{th11}
 Let  $ 0 \leqq \mu \leqq1,\,\,  \vartheta
\geqq 0,\,\,\,\,\,   \kappa \geqq 0 $ \text{and} $\wp_n > 0.$  If  $f\in\mathfrak{G}_{\vartheta,\, \kappa}^{\wp}(\mathcal{X},\varkappa)$  be
given by $(\ref{e1.3})$
then
\begin{eqnarray*}
|a_3-\mu a_2^2| &\leqq & \left\{  \begin{array}{lll}\displaystyle \frac{1}{2 \mathbf{L}\wp_3}
\left( 1+\varkappa+\frac{\aleph_2} { \mathbf{W}^{2}} \right), & \mbox{ if } &
 \mu \leqq \sigma_1, \\[5mm] \displaystyle
\frac{1}{\mathbf{L}\wp_3}, & \mbox{ if } & \sigma_1\leqq \mu \leqq \sigma_2, \\ [5mm]
\displaystyle \frac{-1}{2 \mathbf{L}\wp_3} \left( 1+\varkappa+\frac{\aleph_2} { \mathbf{W}^{2}} \right),
& \mbox{ if } &
\mu \geqq \sigma_2,\\
\end{array}\right.
\end{eqnarray*}
where, for convenience,
\begin{eqnarray}
\nonumber \sigma_1 & := &\frac{\wp_2^{2}}{\wp_3}~\left[ \frac{(\varkappa-1) \mathbf{W} ^2 +2( M\kappa^{2}+ S\kappa+Q)
}
{2 \mathbf{L} }\right],\nonumber \qquad\sigma_2 = \frac{\wp_2^{2}}{\wp_3}\left[\frac{\varkappa \mathbf{W} ^2 +2( M\kappa^{2}+ S\kappa+Q)
}
{2 \mathbf{L} }\right],\\[2mm]
 \nonumber
 \aleph_2& := & 2( M\kappa^{2}+ S\kappa+Q)
 - 2 \mu \mathbf{L}\frac{\wp_3}{\wp_2^{2}} ,
\end{eqnarray}
 $M,S,Q$ are assumed as in \eqref{m} and $\mathbf{L,W}$ are as in \eqref{xi} and \eqref{tau}
\end{theorem}
\begin{proof}
By \eqref{PMG1}, \eqref{PMG2} and proceeding as in Theorems \ref{th1} and \ref{gms}  we get required result.
\end{proof}
\subsection{Application to functions  based on certain distributions}
 A variable $x$ is said to be Poisson distributed  is given by
\begin{equation*}
P(x=r)=\frac{m^{r}e^{-m}}{r!},\text{ }r=0,1,2,3,\cdots ,
\end{equation*}%
where $m$
is called the parameter.
\bigskip
In \cite{por1}, Porwal represented in
power series  given by 
\begin{equation*}
P(m,z)=z+\sum\limits_{n=2}^{\infty }\frac{m^{n-1}}{(n-1)!}%
e^{-m}z^{n},\text{\ \ \ \ \ \ \ \ }z\in \mathcal{U}\text{,}\quad m>0.
\end{equation*}%
By ratio test the radius of convergence of above series is
infinity. Using the convolution, he defined a
linear operator $\mathcal{J}^{m}(z):\mathcal{A\rightarrow A}$ (see also, \cite{por1, SME,mur1,mur2,por2}%
\begin{eqnarray*}
\mathcal{J}^{m}f= f(z)\ast P(m,z) &=& z+\sum\limits_{n=2}^{\infty }%
\psi_m a_{n}z^{n},\\ &=&z+\psi_2a_2z^2+\psi_3a_3z^3+\cdots \qquad z\in \Delta,
\end{eqnarray*}%
where
$\psi_n= \frac{m^{n-1}}{(n-1)!}e^{-m}.$  In~ particular \begin{equation} \label{lp1}\psi_2=m e^{-m}\qquad{\text and}\qquad  \psi_3=\frac{m^2}{2}e^{-m}.\end{equation}

 From $(\ref{lp1})$ , by taking $\wp_2= m e^{-m}=\psi_2$ and $\wp_3=\frac{m^2}{2}e^{-m}=\psi_3$ one can easily state the results( as in Theorems \ref {gms}  and $\ref{th11}$)  associated with Poisson  distribution  .
\par Recently various subclasses of analytic , univalent and by univalent functions  are discussed based on Borel
distribution \cite{SM1,SM2,HM} of a discrete random variable $X$
 with parameter $\varsigma$ with probability mass function
as given by
\begin{equation}\label{sec7e1}
p(X=r)=\frac{(\varsigma r)^{r-1}e^{-\varsigma r}}{r!}\quad r=1,2,3,\cdots.
\end{equation}
Recently, Wanas and Khuttar\cite{wanas} gave a power series
representation
\begin{equation}\label{sec7e2}
\mathfrak{B}(\varsigma,
z)=z+\sum_{n=2}^{\infty}\frac{(\varsigma(n-1))^{n-2}e^{-\varsigma(n-1)}}{(n-1)!}z^{n}
\quad (z \in \Delta, 0 \leq \varsigma \leq 1)
\end{equation}
where  whose coefficients are the probabilities of Borel distribution
. By way of the use of ratio test, it could be shown
that the radius of convergence of the above  series is infinity.
Let us introduce a linear operator $$L_{\varsigma}:
\mathfrak{A}\longrightarrow \mathfrak{A}$$ defined by
\begin{eqnarray}\label{sec7e3}
L_{\varsigma}f(z)=\mathfrak{B}(\varsigma,
z)*f(z)
&=&z+\sum_{n=2}^{\infty}\Lambda_{n}a_{n}z^{n}\nonumber\\
&=&z+\Lambda_{2}a_{2}z^{2}+\Lambda_{3}a_{3}z^{3}+\cdots,
\end{eqnarray}
where
$\Lambda_{n}=\Lambda_{n}(\varsigma)=\frac{(\varsigma(n-1))^{n-2}e^{-\varsigma(n-1)}}{(n-1)!}.$ By fixing $n=2,3$  we have
$\Lambda_{2}=e^{-\varsigma}\quad \Lambda_{3}=\varsigma e^{-2\varsigma}.$
Now by  taking $\wp_2=\Lambda_{2}=e^{-\varsigma}\quad \wp_3=\Lambda_{3}=\varsigma e^{-2\varsigma}$
 as in Theorems \ref {gms}  and $\ref{th11}$ we state the results in association with Borel distribution.

Lately, El-Deeb et al.\cite{pascal,GMSTB} introduced a power series whose coefficients are $$(1-q)^{s},\dfrac{qs(1-q)^{s}}{%
1!},\dfrac{q^{2}s(s+1)(1-q)^{s}}{2!},\dfrac{q^{3}s(s+1)(s+2)(1-q)^{s}}{%
3!}\cdots$$, respectively, probabilities of Pascal
distribution is
\begin{equation*}
\Theta_{q}^{s}(z)=z+\sum\limits_{n=2}^{\infty }\binom{n+s-2}{s-1}
q^{n-1}(1-q)^{s}z^{n},\qquad z\in \mathbb{D} ;s\geq 1;0\leq q\leq 1. \label{PHI}
\end{equation*}%
 whose radius of
convergence  is infinity by ratio test.
Now, we define the linear operator
$
\Lambda_{q}^{s}(z):\mathfrak{A}\rightarrow \mathfrak{A}
$
\begin{eqnarray}
\Lambda_{q}^{s}f(z)=\Theta_{q}^{s}(z)\ast
f(z)
&=& z+\sum\limits_{n=2}^{\infty }\Phi_n a_{n}z^{n},\qquad z\in \mathbb{D}\label{I}\end{eqnarray}
where $\Phi_n=\binom{n+s-2}{s-1}
q^{n-1}(1-q)^{s}.$
          Now by  taking $$\wp_2=\Phi_2=\binom{s}{s-1}
q(1-q)^{s}\quad \wp_3=\Phi_3=\binom{s+1}{s-1}
q^{2}(1-q)^{s}$$
one can easily state the results( as in Theorems \ref {gms}  and $\ref{th11}$)  associated  Pascal distribution.
\section*{Conclusion}We investigated  coefficient estimates  for  $f\in\mathfrak{G}_{\vartheta}^{\kappa}(\mathcal{X},\varkappa)$  analytic
functions subordinating generalized telephone numbers,
and  derived initial coefficient estimates $a_2,a_3$ and Fekete-Szeg\"{o} inequality for $f\in\mathfrak{G}_{\vartheta}^{\kappa}(\mathcal{X},\varkappa)$. A similar results have been  done for the function $ f^{-1} $ and $\log\dfrac{f(z)}{z}.$
 Further application of our results to certain functions defined by convolution products. Appropriately specifying the parameters in Theorems \ref{gmsth1} to \ref{thm1d} it is easy to straightforwardly state
the results for the numerous new subclasses listed in Remark \ref{GME-REM} which can be new and now not discussed thus far by way of subordinating with telephonic numbers.
 In addition we can state results as Theorems \ref {gms}  and $\ref{th11}$ the function classes  connected with Poisson ,Borel  and Pascal distributions which also not discussed so far . Further
keeping with the latest trend of research we can extended the study using
Quantum calculus see\cite{Srivastava-q calculus,SSrivastava-qsurvey}.

\section*{Declaration Statements}

\noindent{\bf Data availability}:  No data were used to support this study.

\medskip

\noindent{\bf Competing interests}: We declare that we do not have any commercial or associative interests that represent conflicts of interest in connection with this manuscript. There are no professional or other personal interests that can inappropriately influence our submitted work.

\medskip

\noindent{\bf Authors' Contributions}: We contributed equally to the writing of this article, and they read and approved the final manuscript for publication.

\medskip

\noindent{\bf Funding}: Not applicable.

\end{document}